\documentclass{amsart}
\usepackage{amssymb}
\usepackage{amsrefs}

\usepackage[all]{xy}
\newtheorem{theorem}{Theorem}[section]
\newtheorem*{theorem*}{Theorem}
\newtheorem{lemma}[theorem]{Lemma}
\newtheorem{corollary}[theorem]{Corollary}
\newtheorem{proposition}[theorem]{Proposition}
\theoremstyle{definition}
\newtheorem{definition}[theorem]{Definition}
\newtheorem{example}[theorem]{Example}
\theoremstyle{remark}
\newtheorem{remark}[theorem]{Remark}

\numberwithin{equation}{section}
\newcommand{\Z}{\ensuremath{\mathbb{Z}}}   

\newcommand{\Q}{\ensuremath{\mathbb{Q}}}

\newcommand{\Ext}{\operatorname{Ext}} 

\newcommand{\Image}{\operatorname{Im}} 
\newcommand{\Ker}{\operatorname{Ker}}

\newcommand{\coclosed}[2]{\ensuremath{\xymatrix@1{ #1 \, \ar@{^{(}->}[r]^-{cc} &
#2}}}
\newcommand{\cosmall}[3]{\ensuremath{\xymatrix@1{ #1 \,
\ar@{^{(}->}[r]^-{cs}_-{#2} & #3}}}

\newcommand{\ShortExactSequence}[5]{\ensuremath{\xymatrix@1{ 0 \ar[r] &  #1 \ar[r]^-{#2} & #3 \ar[r]^-{#4} &  #5  \ar[r] & 0 }}}

\newcommand{\LongExactSequence}[5]{\ensuremath{\xymatrix@1{ 0 \ar[r] & #5 \ar[r] &#1 \ar[r]^-{#2} & #3 \ar[r]^-{#4} &  #5  \ar[r] & 0 }}}

\newcommand{\AShortExactSequence}[6]{\ensuremath{\xymatrix@1{ #1: 0 \ar[r] &  #2 \ar[r]^-{#3} & #4 \ar[r]^-{#5} &  #6  \ar[r] & 0 }}}

\begin{document}

\author{Rafail Alizade}
\address{Ya\c{s}ar University, Department of Mathematics, Izmir, Turkey.}
\email{rafail.alizade@yasar.edu.tr}

\author{Eng\.{I}n B\"{u}y\"{u}ka\c{s}{\i}k}
\address{Izmir Institute of Technology, Department of Mathematics,\\ G\"{u}lbah\c{c}ek\"{o}y\"{u}, 35430, Urla, Izmir, Turkey.}
\email{enginbuyukasik@iyte.edu.tr}


\title{poor and pi-poor abelian groups}
\date{\today}

\begin{abstract}

In this paper, poor abelian groups are characterized. It is proved that an abelian group is poor if and only if its torsion part contains a direct summand isomorphic to $\oplus_{p \in P} \Z_p$, where $P$ is the set of prime integers.  We also prove that pi-poor abelian groups exist.  Namely, it is proved that the direct sum of $U^{(\mathbb{N})}$, where $U$ ranges over all nonisomorphic uniform abelian groups, is pi-poor. Moreover, for a  pi-poor abelian group $M$, it is shown that $M$ can not be torsion, and each $p$-primary component of $M$ is unbounded. Finally, we show that there are pi-poor groups which are not poor, and vise versa.
\end{abstract}


\maketitle

\section{Introduction}

Let $R$ be a ring with an identity element and \emph{Mod-R} be the category of right $R$-modules.
Recall that a right $R$-module $M$ is said to be $N$-injective (or injective relative to $N$) if for every submodule $K$ of $N$ and every morphism $f: K \to M$ there exists a morphism $\overline{f}: N \to M$ such that $\overline{f}| _K=f.$ For a module $M$, as in \cite{AF}, the injectivity domain of $M$ is defined to be the collection of modules $N$ such that $M$ is $N$-injective, that is, $\mathfrak{In}^{-1}(M)=\{N \in Mod -  R | \text{$M$ is $N$-injective} \}$. Clearly, for any right $R$-module $M$, semisimple modules in \emph{Mod- R} are contained in $\mathfrak{In}^{-1}(M)$, and $M$ is injective if and only if $\mathfrak{In}^{-1}(M)=$\emph{Mod-R}.  Following \cite{poorSergio}, $M$ is called $poor$ if for every right $R$-module $N$, $M$ is $N$-injective only if $N$ is semisimple, i.e. $\mathfrak{In}^{-1}(M)$ is exactly the class of all semisimple right $R$-modules. Poor modules exists over arbitrary rings, see \cite[Proposition 1]{poorNE}. Although  poor modules exist over arbitrary rings,  their structure is not known  over certain rings including also the ring of integers.

A right $R$-module $N$ is \emph{pure-split} if every pure submodule of $N$ is a direct summand.  Let $K$ and $N$ be right $R$-modules. $K$ is \emph{$N$-pure-injective} if for each pure submodule $L$ of $N$  every homomorphism $f: L \to K$ can be extended to a homomorphism  $g: N \to K$.
Following \cite{pi-poorSergio},  a right $R$-module  $M$ is called \emph{pure-injectively poor (or simply pi-poor)} if whenever $M$ is $N$-pure-injective, then $N$ is pure-split. It is not known whether pi-poor modules exists over arbitrary rings. In particular in  \cite{pi-poorSergio} some classes of abelian groups that are not pi-poor are given but the authors  point out that they do not know  whether a pi-poor abelian group exists.

The purpose of this paper is to give a characterization of poor abelian groups, and also to prove that pi-poor abelian groups exists.

 Namely, in section 2, we prove that an abelian group $G$ is poor if and only if the torsion part of $G$ contains a direct summand isomorphic to $\oplus _{p\in P} \Z_p$, where $ P$ is the set of prime integers (Theorem \ref{poorabelian groups}).

Section 3 is devoted to the  proof of the existence of pi-poor abelian groups. Let $\{A_{\gamma}   | \gamma \in \Gamma \}$ be a complete set of representatives of isomorphism classes of reduced uniform groups. We prove that the group  $M= \bigoplus _{\gamma \in \Gamma} A_{\gamma}^{(\mathbb{N})}$ is  pi-poor (Theorem \ref{pi-poor exist}). In addition, it is proved that if $G$ is a pi-poor abelian group then $G$ is not torsion, and the $p$-primary component $T_p(G)$ of $G$ is unbounded for each prime $p$.

\section{Definitions and Preliminaries}

We recall some definitions and  results  which will be useful in the sequel. For more details we refer the reader to \cite{FuchsAbelvol1}. By group we will mean an abelian group throughout the paper. Let $p \in P$ be a prime integer.  A group $G$ is called \emph{p-group} if every nonzero element of $G$ has order $p^n$ for some $n \in \Z ^+$.
For a group $G$, $T(G)$ denote the torsion submodule of $G$. The set  $T_p(G)= \{ a \in G | p^{k} a=0\,\, \text{for some }\,\, k \in \Z^+ \}$ is a subgroup of $G$, which is called  the \emph{p-primary component} of $G$. For every torsion group $G$, we have $G=\oplus _{p \in P} T_p(G)$.
A subgroup $A$ of a group $B$ is \emph{pure in $B$} if $nA=A \cap nB$ for each integer $n$. A monomorphism (resp. epimorphism) $\alpha : A \to B$ of abelian groups is called \emph{pure} if $\alpha (A)$ (resp. $\Ker (\alpha)$) is pure in $B$.
For any group $G$, the subgroups $T(G)$ and $T_p(G)$ are pure in $G$.
A group $G$ is said to  be \emph{bounded} if $nG=0$, for some nonzero integer $n$.
Bounded groups are direct sum of cyclic groups \cite[Theorem 17.2]{FuchsAbelvol1}. An group $G$ is called a \emph{divisible group} if $nG=G$ for each positive integer $n$. A group $G$ is called a \emph{reduced group} if $G$ has no proper divisible subgroup. Note that, since $\Z$ is noetherian, every group $G$ contains  a largest divisible subgroup. Therefore $G$ can be written as $G=N\oplus D$, where $N$ is reduced and $D$ is divisible subgroup of $G$.

\begin{definition}(see \cite{FuchsAbelvol1}) Let $p \in P$. A subgroup $B$ of a group $A$ is called a \emph{p-basic subgroup of B} if it satisfies the following three conditions:

\begin{enumerate}
\item[(i)] $B$ is a direct sum of cyclic $p$-groups and infinite cyclic groups;
\item[(ii)] $B$ is $p$-pure in $A$ i.e. $pA=A \cap pB$;
\item[(iii)] $A/B$ is $p$-divisible, i.e. $p(A/B)=A/B$.
\end{enumerate}
\end{definition}

\begin{lemma}\noindent
\begin{enumerate}
\item[(a)]\cite[Theorem 32.3]{FuchsAbelvol1} Every group  $G$ contains a $p$-basic subgroup for each $p \in P$.
\item[(b)]\cite[Theorem 27.5]{FuchsAbelvol1} If $H$ is a pure and bounded subgroup of a group $G$, then $H$ is a direct summand of $G$.
\end{enumerate}
\end{lemma}

For $q\neq p$ $q$-basic subgroups of $p$-groups are 0, so only $p$-basic subgroups of $p$-groups may be nontrivial. Therefore they are usually called simply basic subgroups. Clearly basic subgroups of $p$-groups are pure.
Subgroups  of the group of the rational integers $\Q$ are called  \emph{rational}  groups. Let $A$ be a uniform  group. Then, it is easy to see that, either  $A$ is isomorphic to a rational group or $A\cong \Z_{p^n}$, for some $p \in P$ and $n \in \Z^+$. For a torsion-free group $G$, we shall denote the (torsion-free) rank (=uniform dimension) of $G$ by $r_0(G)$ (see, \cite{FuchsAbelvol1}). By \cite[page 86, Ex. 3]{FuchsAbelvol1}, $r_0(G)=r_0(H) + r_0(G/H)$ for each subgroup $H$ of $G$.
A torsion-free group $G$ is said to be \emph{completely decomposable} if $G= \oplus _{i \in I}K_i$, where $I$ is an index set and each $K_i$ is isomorphic to a rational group, i.e. $r_0(K_i)=1$ for each $i \in I$.

\section{poor abelian groups}

In this section we give a characterization of poor groups. In \cite{poorSergio}, the authors prove that the  group $\oplus_{p \in P} \Z_p$ is poor. The following result shows that, this group is crucial in investigation of poor groups.

\begin{theorem}\label{poorabelian groups} A group is poor if and only if its torsion part has a direct summand isomorphic to $\oplus_{p \in P} \Z_p$.
\end{theorem}

\begin{proof} To prove the necessity, let $G$ be a poor group and let $p$ be any prime. If $T_p(G)= 0,$ then $G$ is $N$-injective for every $p$-group $N$, therefore $T_p(G) \neq 0$. If every element of order $p$ of $G$ is divisible by $p$, then $G$ is $\Z _{p^2}$-injective since $\Z_{p^2}$ has only one nontrivial subgroup: $p\Z_{p^2}$. So there is at least one element $a_p$ with $|a_p|=p$, that is not divisible by $p$. Then the cyclic group $<a_p>$ is a $p$-pure subgroup of $T_p(G)$, therefore a pure subgroup of $T_p(G)$. Since bounded pure subgroups are direct summands, $<a_p>$ is a direct summand of $T_p(G)$. Hence $\oplus_{p \in P}<a_p>$ is a direct summand of $\oplus_{p \in P}T_p(G)=T(G)$. Clearly $\oplus_{p \in P}<a_p> \cong \oplus_{p \in P} \Z_p$.

Conversely suppose that $T(G)$ contains a direct summand isomorphic to $\oplus \Z_p$. Let $V$ be a direct summand of $T(G)$ such that $V \cong \Z _p$. Then $V$ is pure in $G$, because $T(G)$ is pure in $G$. So $V$ is a direct summand in $G$ by  \cite[Theorem 27.5]{FuchsAbelvol1}. This implies, for each prime $p$, $G$ contains a  direct summand isomorphic to $\Z_p$.  Now, suppose $G$ is $N$-injective for some group $N$. Then $\Z_p$ is $N$-injective for each prime $p$. Suppose that $N$ is not semisimple (not elementary in terminology of \cite{FuchsAbelvol1}). Then there is an element $a$ of infinite order or with $o(a) = p^{n}$ where $p$ is a prime and $n> 1$. In first case $\langle a \rangle = \mathbb{Z}$ and in second case $\langle a \rangle = \mathbb{Z}_{p^{n}}$. So $\mathbb{Z}_{p}$ must be $\mathbb{Z}$- injective or $\mathbb{Z}_{p^{n}}$-injective by \cite[Proposition 1.4]{ContinuousAndDiscreteModules}. But the homomorphism $f: p\mathbb{Z} \rightarrow \mathbb{Z}_{p}$ with $f(p) = 1$ cannot be extended to $g: \mathbb{Z} \rightarrow \mathbb{Z}_{p}$ since otherwise $1= f(p) = g(p) = pg(1)= 0$ and  $\mathbb{Z}_{p}$ is isomorphic to the subgroup $\langle p^{n-1} \rangle$ of $\mathbb{Z}_{p^{n}}$ which is not a direct summand of $\mathbb{Z}_{p^{n}}$. So in both cases we get a contradiction, that is $N$ is semisimple.
  \end{proof}

The following is a consequence of Theorem \ref{poorabelian groups}.

\begin{corollary} For a group $G$, the following are equivalent.
\begin{enumerate}
\item[(1)] $G$ is poor.
\item[(2)] The reduced part of $G$ is poor.
\item[(3)] $T(G)$ is poor.
\item[(4)] For each prime $p$, $G$ has a direct summand isomorphic to $\Z_p$.

\end{enumerate}
\end{corollary}

\section{pi-poor abelian groups}

In \cite{pi-poorSergio}, the authors investigate the notion of pi-poor module and study properties of these modules over various rings. In particular they give some classes of groups that are not pi-poor  and point out that they do not know whether a pi-poor  group exists or not. In this section we shall prove that pi-poor  groups exist.

\begin{theorem}\label{pi-poor exist}Let $\{A_{\gamma}   | \gamma \in \Gamma \}$ be a complete set of representatives of isomorphism classes of uniform groups. Then the group  $$M= \bigoplus _{\gamma \in \Gamma} A_{\gamma}^{(\mathbb{N})}$$ is  pi-poor.

\end{theorem}

Before proving the theorem, we will first give some lemmas. Throughout this section $M$ denotes  the group given in Theorem \ref{pi-poor exist}.

The following result is well known. We include it for completeness.

\begin{lemma} Let $R$ be a ring and $L$, $N$ be right $R$-modules. Let $K$ be a pure submodule of $N$. If $L$ is $N$-pure-injective, then $L$ is both $K$-pure-injective and $N/K$-pure-injective.
\end{lemma}

\begin{proof} Let $A$ be a pure submodule of $K$ and $f: A \to L$ be a homomorphism. Then $A$ is pure in $N$, and so $f$ extends to a map $g: N \to L$. It is clear that $g | _K: K \to L$ is an extension of $f$ to $K$. Hence $L$ is $K$-pure-injective.
Now, let $X/K$ be a pure submodule of $N/K$ and $f: X/K \to L$ be a homomorphism. Since $K$ is pure in $N$ and $X/K$ is pure in $N/K$, $X$ is pure in $N$. Therefore there is a homomorphism $g: N \to L$ such that $f\pi ' =gi$, where $i: X \to N$ is the inclusion and $\pi ' :X \to X/K$ is the usual epimorphism.
Since $g(K)=0$, $\Ker (\pi) \subseteq \Ker (g)$, where $\pi : N\to N/K$ is the usual epimorphism. Therefore there is a homomorphism $h: N/K \to L$ such that $h\pi =g$. Then for each $x \in X$, $h(x+K)=h (\pi (x))=g(x)=(f \pi')(x)=f(x+K)$. That is, $h$ extends $f$. Hence $L$ is $N/K$-pure-injective.

\end{proof}

\begin{lemma}\label{torsion pure inj} Let $G$ be a reduced torsion  group. The following are equivalent.

\begin{enumerate}
\item[(1)] $M$ is $G$-pure-injective.
\item[(2)] $T_{p}(G)$ is bounded for each $p \in P$.
\item[(3)] $G$ is pure-split.

\end{enumerate}
\end{lemma}

\begin{proof} $(1) \Rightarrow (2)$ Write $G= \oplus_{p \in P} T_p(G)$.  Let $B_p(G)$ be a basic subgroup of $T_p(G)$.  Then $B_p(G)$ is pure in $T_p (G)$, and so in $G$ and $T_p(G)/B_p(G)$ is divisible. We claim that, $B_p(G)$ is bounded. Suppose the contrary that $B_p(G)$ is not bounded. Then for every positive integer $n$, $B_p(G)$ contains an element of order $p^n$. In this case, since $B_p(G)$ is a direct sum of cyclic $p$-groups, there is an epimorphism $$B_p(G) \overset{g}{\to} \Z_{p^{\infty}} \to 0,$$ where the restrictions of $g$ to  the cyclic summands of $B_p(G)$ are monic. It can be proved as in \cite[Lemma 30.1]{FuchsAbelvol1} that  $g$ is a pure epimorphism, i.e. $K=\Ker (g)$ is a pure submodule of $B_p(G)$.  Now $K$ is pure in $B_p(G)$ and is a direct sum of cyclic $p$-groups. Since $M$ contains a direct summand isomorphic to $K$, and $B_p(G)$ is a pure subgroup of $G$,  $K$ is  $B_p(G)$-pure-injective. Therefore $B_p(G) \cong K \oplus \Z_{p^{\infty}}$. This contradicts with the fact that $B_p(G)$ is reduced. Hence $B_p(G)$ is bounded, and so $B_p(G)$ is a direct summand of $G$. The fact that $G$ is reduced and $T_p(G)/B_p(G)$ divisible implies that $B_p(G)=T_p(G)$.

$(2) \Rightarrow (3)$ Let $H$ be a pure subgroup of $G$. Since $G= \oplus_{p \in P} T_p(G)$  and  $H=\oplus_{p \in P} T_p(H)$, $T_p(H)$ is a pure subgroups of $T_p(G)$. Then $T_p(H)$ is a direct summand of $T_p(G)$ by \cite[Theorem 27.5]{FuchsAbelvol1}. Let $T_p(G)=T_p(H)\oplus N_p$, where $N_p \leq G$. Then $G=\oplus_{p \in P}[T_p(H)\oplus N_p]=(\oplus_{p \in P} T_p(H))\oplus (\oplus_{p \in P} N_p)=H \oplus(\oplus_{p \in P} N_p)$. Hence $G$ is pure-split.

$(3) \Rightarrow (1)$ Clear by the definition.
\end{proof}

\begin{remark}Pure-split  groups are completely  characterized  in \cite{CompletelyDecomposableGroups}. The implications $(2)\Leftrightarrow (3)$ in Lemma \ref{torsion pure inj} also can be found in \cite{CompletelyDecomposableGroups}.

\end{remark}

\begin{lemma}\label{lemma:torsion-puresplitting}Let $B$ be a $p$-group. Suppose that $M$ is $B$-pure-injective. Then $B$ is pure-split.
\end{lemma}

\begin{proof} Let $D$ be the divisible subgroup of $B$ and $A$ be a pure subgroup of $B$. Then  $B=C\oplus D$ for some reduced group $C$. Let $D_A$ be the divisible subgroup of $A$. Then $D_A\leq D$, and $D=D_1\oplus D_A$ for some $D_1\leq D$. So $B=C\oplus D_1\oplus D_A=E\oplus D_A$ where $E=C\oplus D_1$. By modular Law $A=(E\cap A)\oplus D_A$. Then  $L=E\cap A$ is a pure submodule of $B$. Hence $M$ is $L$-pure-injective, and $L\cong A/D_A$ is reduced. Therefore $L$ is bounded by Lemma \ref{torsion pure inj}.  Since $L$ is pure in $B$, $L$ is also pure in $E$. Then $E=K\oplus L$ for some $K\leq E$ by \cite[Theorem 27.5]{FuchsAbelvol1}. Then $B=E\oplus D_A=K\oplus L\oplus D_A=K\oplus A$. So $A$ is a direct summand in $B$. Hence $B$ is pure-split.
\end{proof}

\begin{lemma}\label{lemma:reduced torsionfreepure split} If $N$ is a reduced torsion-free group such that $M$ is $N$-pure-injective then $N$ is pure-split. Moreover, $N$ is completely decomposable with finite rank.
\end{lemma}

\begin{proof} Take any $0\neq a_1\in N$ and let $G_1=\{x\in N \,| \, mx\in \langle a_1\rangle\, \text{for some} ~0\neq m\in\mathbb{Z}\}$ (that is $G_1$ is the subgroup purely generated by $a_1$). Clearly $G_1$ is a pure subgroup of $N$ and isomorphic to a rational group, so $M$ has a direct summand isomorphic to $G_1$. Therefore $G_1$ is a direct summand of $N$, that is $N=G_1\oplus N_1$ for some $N_1\leq N$. If $N_1\neq 0$ we can find in similar way a pure subgroup $G_2$ of $N_1$ purely generated by an element $a_2$. Clearly $M$ is $N_1$-pure-injective, so $N_1=G_2\oplus N_2$. The same can be done for $N_2$ if $N_2\neq 0$ and so on. If this process continuous infinitely then $N$ contains a subgroup $\oplus_{i=1}^\infty G_i$ which is pure as a direct limit of pure subgroups. Therefore $M$ is $\oplus_{i=1}^\infty G_i$-pure-injective. For each $a_i$, $i= 1,2,\ldots$ there is a homomorphism $f_i: \langle a_i\rangle \rightarrow \mathbb{Q}$ with $f(a_i)=\frac{1}{i}$. Since $\mathbb{Q}$ is injective there is a homomorphism $f:\oplus_{i=1}^{\infty} G_i\rightarrow \mathbb{Q}$ with $f(a_i)=f_i(a_i)=\frac{1}{i}$. Clearly $f$ is an epimorphism. Since $\mathbb{Q}$ is torsion-free, $K=\Ker(f)$ is a pure subgroup of $\oplus_{i=1}^\infty G_i$. Let $\Gamma$ be the set of all completely decomposable pure subgroups of $K$ and $R$ be the set of all subgroups of $K$ of rank $1$. Define order $\preceq$ on $\Gamma$ as follows: $\oplus_{S\in I}S\preceq\oplus_{S\in J}S$ if $I\subseteq J\subseteq R$. If $P$ is any chain in $\Gamma$, then $\cup_{X\in P}X$ is clearly a completely decomposable and pure subgroup of $K$, since the direct limit of pure subgroups is pure. So by Zorn's Lemma there is a maximal element $B=\oplus_{S\in T}S$ in $\Gamma$. Since $K$ is countable $T$ is also countable, so $B$ is a direct summand of $K$, that is $K=B\oplus C$ for some $C\leq K$. If $C\neq 0$ then as at the beginning of the proof we can find a pure subgroup of $X$ of $C$ of rank $1$. Clearly $B\oplus X\in \Gamma$. Contradiction with maximality of $B$. So $C=0$. Then $K$ is a direct summand of $\oplus_{i=1}^\infty G_i$. So $\oplus_{i=1}^\infty G_i \cong K\oplus \mathbb{Q}$. But $\oplus_{i=1}^\infty G_i$ is reduced. Contradiction. Thus the process must be finite, that is $N=G_1\oplus G_2\oplus\cdots \oplus G_n$ for some $n\in \mathbb{Z}^{+}$. To show that $N$ is pure-split let $L$ be a pure subgroup of $N$. Then $M$ is $L$-pure-injective, so it is the direct sum of groups of rank one of finite number as we have proved above. Then $L$ is a direct summand of $N$, because  $N$-pure-injectiveness of $M$ implies that the inclusion $L \to N$ is splitting. Hence $N$ is pure-split and completely decomposable with finite rank. This completes the proof.
\end{proof}

\begin{lemma}\label{lemma:torsionFree pure splitting} Let $N$ be a torsion-free group. If $M$ is $N$-pure-injective, then $N$ is pure-split.
\end{lemma}

\begin{proof} Let $K$ be a pure subgroup of $N=A \oplus D$, where $D$ is the divisible subgroup of $N$. Let $D_K$ be the divisible subgroup of $K$. Then $D_K \leq D$, and so $D=D_1\oplus D_K$ for some $D_1\leq D$. So  $N=A\oplus D_1\oplus D_K=E\oplus D_K$, where $E=A\oplus D_1$. By modular law, $K=(E\cap K)\oplus D_K$. Denote $E\cap K=L$. Then $L\cong K/D_K$ is reduced and pure in $N$. Hence $M$ is $L$-pure-injective, and so  $L \cong \oplus _{i=1}^n R_i$ for some rational groups $R_1,\ldots R_n$, by Lemma \ref{lemma:reduced torsionfreepure split}. Then $M$ contains a direct summand isomorphic to $L$. So the inclusion $L \to N$ splits, i.e. $N=L \oplus H$ for some $H \leq N$. Since $L$ is reduced, $D_K \leq H$. Then $N=L \oplus D_K \oplus H'=K \oplus H'$. This implies that $N$ is pure-split.
\end{proof}

\begin{definition}(See, \cite{FuchsVol.2}) Let $G$ be a torsion-free group and  $a \in G$. Given a prime $p$, the largest integer $k$ such that $p^k | a$ holds is called the \emph{p-height} $h_p(a)$ of $a$; if no such maximal integer $k$ exists, then we set $h_p(a)=\infty$. The sequence of $p$-heights $$\chi(a)=(h_{p_1}(a), h_{p_2}(a),\cdots , h_{p_n}(a),\cdots)$$is said to be the characteristic of $a$. Two characteristics $(k_1,k_2,\cdots)$ and $(l_1,l_2,\cdots)$ are \emph{equivalent} if $k_n \neq l_n$ holds only for a finite number of $n$ such that in case $k_n \neq l_n$ both $k_n$ and $l_n$ are finite. An equivalence class of characteristics is called a \emph{type}. $G$ is called \emph{homogeneous} if all nonzero elements of $G$ are of the same type.
\end{definition}

\begin{corollary} Let $N$ be a torsion-free reduced group. The following are equivalent.

\begin{enumerate}
\item[(1)] $M$ is $N$-pure-injective.
\item[(2)] $N$ is pure-split.

\item[(3)] $N$ is a completely decomposable homogeneous group of finite rank.
\end{enumerate}
 \end{corollary}

\begin{proof} $(1) \Leftrightarrow (2) $ By Lemma \ref{lemma:reduced torsionfreepure split}.

$(2) \Leftrightarrow (3)$ See \cite{CompletelyDecomposableGroups} or \cite[Ex.8, page 116]{Fuchs Vol.2}.

\end{proof}

Now we can prove our theorem.

\vspace{0.5cm}

\noindent {\bf Proof of Theorem \ref{pi-poor exist}.} Let $M$ be $G$-pure-injective for some group $G$. We have $G=N \oplus D$ for some reduced group $N$ and a divisible group $D$. Then  $M$ is $N$-pure-injective, and since $T(N)$ is a pure subgroup of $N,$ $M$ is $T(N)$-pure-injective and $M$ is $N/T(N)$-pure-injective. Then, by Lemma \ref{torsion pure inj} and Lemma \ref{lemma:reduced torsionfreepure split}, $T(N)=\oplus_{p \in P}B_p(N)$ and $N/T(N)=\oplus_{i \in I}K_i$, where for  each $p \in P$, $B_p(N)$ is a bounded $p$-group, $I$ is a finite index set and each $K_i$ is isomorphic to a rational group.  We claim that $T(N)$ is a direct summand in $N$, that is, the short exact sequence $$\mathbb{E}: 0 \to T(N) \to N \to N/T(N)\to 0$$is splitting. By \cite[Theorem 52.2]{FuchsAbelvol1} there is a natural isomorphism $$\Ext(N/T(N), T(N))=\Ext(\bigoplus _{i\in I}K_i, T(N)) \cong \prod_{i \in I} \Ext (K_i, T(N)) $$ induced by the inclusions $\alpha _j :K_j \to \oplus_{i \in I}K_i$. Therefore it is sufficient to prove that each short exact sequence $$\mathbb{E} \alpha _ j: 0 \to T(N) \to N' \overset{f} \to K_j \to 0$$ is splitting. We have the following commutative diagram with exact columns and rows.

$$
\begin{xy}
  \xymatrix{
   & &  & 0 \ar[d] & 0 \ar[d] & & \\
 &\mathbb{E}: 0 \ar[r] & T(N) \ar@{=}[d] \ar[r] & N' \ar[d] \ar[r]^{f} & K_j \ar[d]^{\alpha _j} \ar[r] & 0 & \\
   &\mathbb{E} \alpha _ j: 0 \ar[r] & T(N) \ar[r] & N\ar[r] \ar[d] & \oplus_{i \in I}K_i \ar[r]\ar[d] & 0 &\\
   &  &   &  \oplus_{i \neq j}K_i \ar@{=}[r] \ar[d]  & \oplus_{i \neq j}K_i \ar[d] &
   &\\
   & &  & 0 & 0 & &
 }
\end{xy}
$$
Since $\oplus_{i \in I}K_i$ is torsion-free, $N'$ is a pure subgroup of $N$, therefore $M$ is $N'$-pure-injective. There is a countable set $\{n_k | k=1,2,\cdots \}$ in $N'$ such that the elements $f(n_k)$ generate $K_j.$ By \cite[Proposition 26.2]{FuchsAbelvol1}, there is a countable pure subgroup $L$ of $N'$ containing the subgroup $\sum_{k=1}^{\infty} \Z n_k.$ Then $M$ is $L$-pure-injective as well. Clearly $f(L)=K_j$ and $\Ker (f| _L)=T(L)$. Since L is countable, $T(L)$ is a countable subgroup of $T(N)$. But $T(N)$ is a direct sum of cyclic primary groups, therefore $T(L)$ is a countable direct sum of  cyclic primary groups and hence is isomorphic to a direct summand of $M$. Since $T(L)$ is a subgroup of $L$ and $M$ is $L$-pure-injective, $T(L)$ is a direct summand of $L$. We have the following commutative diagram with exact rows:

$$
\begin{xy}
  \xymatrix{
 & & &\mathbb{E'}: 0 \ar[r] & T(L)\ar[d]^{\beta} \ar[r] &  L \ar[d]\ar[r] & K _j \ar[r] \ar@{=}[d] & 0 \\
& & &\mathbb{E}\alpha _j : 0 \ar[r] & T(N) \ar[r] & N'\ar[r] & K_j \ar[r] & 0 & &
              }
\end{xy}
$$
where $\beta$ is the inclusion. Since $\mathbb{E'}$ is splitting $\mathbb{E}\alpha _j =\beta \mathbb{E}$ is also splitting. So $N=T(N) \oplus K$, where $T(N)$ and $K$ are groups as in Lemma \ref{torsion pure inj} and Lemma \ref{lemma:reduced torsionfreepure split}, respectively. This proves our claim.

To prove that $G$ is pure-split take a pure subgroup $A$ of $G$. By the first part of the proof, we have $$G=N\oplus D=T(N)\oplus K \oplus T(D) \oplus D'= (T(N) \oplus T(D))\oplus (K \oplus D')= T(G)\oplus G'.$$ Then for each $p \in P$, $T_p(A)$ is a pure subgroup of $T_p(G)$. Therefore $T_p(A)$ is a direct summand of $T_p(G)$ by Lemma \ref{lemma:torsion-puresplitting}. Then $T(A)$ is a direct summand of $T(G)$. We have a homomorphism $f:A/T(A) \to G/T(G)$ defined by $f(a+T(A))=a +T(G)$. If $f(a+T(A))=0$ then $a \in T(G) \cap A=T(A)$, hence $a +T(A)=0$, so $f$ is a monomorphism. Now claim that $Im(f)$ is a pure subgroup of $G/T(G)$. To show this, let $a+T(G)=m(b+T(G))$ for some $a \in A$, $b \in G$, $0 \neq m \in \Z$. Then $a-mb \in T(G)$, therefore $ka=kmb$ for some $0 \neq k \in \Z$. Since $A$ is pure in $G$, $ka=kma'$ for some $a' \in A$. Then $a-ma'\in T(A)$, hence $a+T(A)=m(a'+T(A))$. So $Im(f)$ is pure. Since  $G/T(G) \cong G'$ is pure-split by Lemma \ref{lemma:torsionFree pure splitting}, $f$ is splitting. As $A$ is a pure subgroup of $G$, $M$ is $A$-pure-injective. So again by the first part of the proof $A=T(A) \oplus K'$ for some $K' \leq A$ with $K' \cong A/T(A)$. Then the inclusion map $A=T(A)\oplus K' \to G=T(G)\oplus G'$ is splitting, that is, $A$ is a direct summand in $G$. This completes the proof.

\section{structure of pi-poor abelian groups}

In this section, we prove some results concerning a possible structure of pi-poor groups.

\begin{proposition}\label{prop. pi-poor torsion Tp unbounded} If $G$ is pi-poor group, then $T_p(G)$ is unbounded for each $p \in P$.

\end{proposition}

\begin{proof} Suppose $G$ is pi-poor and  $T_p(G)$ is bounded for some $p \in P$. Then $T_p(G)$ is pure-injective and $T_p(G)$ is a direct summand of $G$, because $T_p(G)$ is also pure in $G$.  Consider the group $\oplus_{n =1} ^{\infty}\Z _{p^n}$. We claim that, $G$ is $\oplus_{n =1} ^{\infty}\Z _{p^n}$-pure-injective. Let $H$ be a pure subgroup of $\oplus_{n =1} ^{\infty}\Z _{p^n}$ and  $f: H \to G$ be a homomorphism. Since $H$ is a $p$-group, $f(H) \subseteq T_p(G).$ So that $f$ extends to a homomorphism $h: \oplus_{n =1} ^{\infty}\Z _{p^n} \to G$, because $T_p(G)$ is pure-injective. This proves our claim.

We shall see that $\oplus_{n =1} ^{\infty}\Z _{p^n}$ is not pure-split. There is an exact sequence $$\mathbb{E}:0\to K \to \oplus_{n =1} ^{\infty}\Z _{p^n} \overset{g} \to \Z_{p^{\infty}}\to 0.$$By the same arguments as in the proof of Lemma \ref{torsion pure inj}, $\mathbb{E}$ is pure, i.e. $K$ is pure in $\oplus_{n =1} ^{\infty}\Z _{p^n}$. Since $\oplus_{n =1} ^{\infty}\Z _{p^n}$ is reduced, $\mathbb{E}$ does not split. Hence $\oplus_{n =1} ^{\infty}\Z _{p^n}$ is not pure-split. This contradicts with the fact that $G$ is pi-poor. Therefore, $T_p(G)$ can not be bounded.
\end{proof}

Let $\Q _p$ be the localization of $\Z$ at the prime ideal $p\Z$. Note that the elements of $\Q _p$ are of the form $ab^{-1}$, where $a, b \in \Z$, $b\neq 0$ and $\gcd (b,p)=1$

\begin{lemma}\label{Qp} Let $p$ be a prime integer and $N$ be a reduced torsion group. Then for every homomorphism $f: \mathbb{Q}_{p} \rightarrow N$, $\Image f$ is bounded.
\end{lemma}

\begin{proof} For every prime $q \neq p,$ it is clear that $q\Q _p=\Q_p$ i.e. $\Q_p$ is $q$-divisible, and $T_q(N)$ is reduced. Then for $\pi _q \circ f: \Q_p \to T_q (N)$ , where $\pi_q: N \to T_q(N)$ is the natural projection, $(\pi_q\circ f)(\Q _p)$ is a $q$-divisible subgroup of $T_q(N)$. Therefore  $(\pi_q \circ f)(\Q _p)$ is  divisible, and so $\pi_q \circ f=0$, because $T_q (N)$ is reduced. Thus $\Image f = f(\Q_p) \subseteq T_p(N)$. Put $a=f(1)$ and $o(a) = p^n$, where $o(a)$ the order of $a$. Let $bc^{-1}$ be any element of $\Q_p$ with $\gcd (c,p)=1$. Then $\gcd (c, p^n)=1,$ therefore $cy +p^n z=1$ for some $y,z \in \Z$. Now $b=bcy+bp^nz,$ so $bc^{-1}=by +bp^n zc^{-1}.$ Note that, $cf(bp^nzc^{-1})=bzp^nf(1)=zp^na=0$. Let $x=f(bp^nzc^{-1})$ and $o(x) = p^m$. Since $\gcd (c,p^m)=1,$ we have $cu + p^mv=1$ for some $u,v \in \Z$. Then $x=ucx +vp^m x=ucx=0$, and so $f(bc^{-1})=f(by)+x=f(by)=byf(1) \in  \langle f(1) \rangle$. Hence $\Image f $ is contained in $ \langle f(1) \rangle$, and so it  is bounded.
\end{proof}

A \emph{cotorsion} group $G$ is a group satisfying $\Ext(\Q, G)=0$.
\begin{theorem} There is a group $G$ such  that $G$ is not pure-split and every reduced torsion group $N$ is $G$-pure-injective. Hence a pi-poor group can not be torsion.

\end{theorem}

\begin{proof} Fix any prime $p$.  Since $\Q_p$ is not cotorsion, $\Ext (\Q, \Q_p)\neq 0$ (see, \cite{FuchsAbelvol1}, p.226, Ex.15). So there is a non splitting pure sequence $$0 \to \Q _p \to G \to \Q \to 0.$$ Hence $G$ is not pure-split. For every prime $q\neq p$, $\Q_p$ and $\Q$ are $q$-divisible, therefore $G$ is also $q$-divisible. We claim that, $N$ is $G$-pure injective. Without loss of generality we can assume that $\Q_p$ is a subgroup of $G$ and $G/\Q_p = \Q$.
Let $K$ be any nonzero pure subgroup of $G$ and $f:K \rightarrow N$ be any homomorphism where $N$ is a torsion reduced group. Then $K$ is $q$-divisible for every prime $q\neq p$ since $K$ is a pure subgroup of $G$ and $G$ is $q$-divisible.
Clearly the rank of $K$ is at most 2. So have two cases:

{\bf Case I:} $r_0(K)=1$.  If $K$ is also $p$-divisible, then $K$ is divisible. So $K\cong \Q$, and the inclusion  $K \to G$ splits, so $f$ can be extended to a homomorphism $f': G\rightarrow N$. Now let $K$ be not $p$-divisible. $K$ and $\Q_p$ are of the same type, and so $K \cong \Q_p$ (see, \cite[Theorem 85.1]{FuchsAbelvol1}). Therefore $\Image f$ is bounded by Lemma \ref{Qp}. Then $\Image f$ is pure-injective, hence $f:K \to N$ can be extended to a homomorphism $f': G \to \Image f \leq N$.

{\bf Case II:} $r_0(K)= 2$: We claim that $K=G$. Otherwise, since $G/K$ is a nonzero torsion-free group,  $r_0(G/K)\geq 1$. Then  $2=r_0(G)=r_0(K) + r_0(G/K)>2$,  a contradiction. Hence $G=K$.

As a consequence, $N$ is $G$-pure-injective. This implies that $N$ is not pi-poor.

\end{proof}

\begin{corollary} Let $M$ be a pi-poor group. Then $M\neq T(M)$  and $T_p(M)$ is unbounded for every $p \in P$.

\end{corollary}

\begin{lemma}\label{lemma:M +N poor iff M poor} Let $M$ and $N$ be right $R$-modules. Assume that $N$ is (pure-)injective. Then $M \oplus N$ is (pi-)poor if and only if $M$ is (pi-)poor.
\end{lemma}

\begin{proof}For a right $R$-module $B$, it is clear that $M\oplus N$ is $B$-(pure-)injective if and only if $M$ is $B$-(pure-)injective.
\end{proof}

\begin{example}
Let  $G= \oplus_{p \in P}\Z_p$. Then $G$ is poor by Theorem \ref{poorabelian groups}. On the other hand, since $T_p(G)=\Z_p$ is bounded,  $G$ is not pi-poor by Proposition \ref{prop. pi-poor torsion Tp unbounded}.
\end{example}

\begin{example} Let $M$ be as in Theorem \ref{pi-poor exist} and let $V$ be the sum of all  direct summands isomorphic to $\Z _p$. If $M=V \oplus K$, then $K$ is pi-poor by Lemma \ref{lemma:M +N poor iff M poor}. But $K$ is not poor by Theorem \ref{poorabelian groups}, since $K$ does not contain a direct summand isomorphic to $\Z_p$. So pi-poor modules need not be poor.

\end{example}


\begin{bibdiv}
\begin{biblist}




\bib{poorSergio}{article}{AUTHOR ={Alahmadi, A. N.}, AUTHOR={Alkan, M.}, AUTHOR={ L{\'o}pez-Permouth,
              S.}, TITLE = {Poor modules: the opposite of injectivity}, JOURNAL = {Glasg. Math. J.},

    VOLUME={52},
      YEAR= {2010},
    NUMBER = {A},
     PAGES = {7--17},}


\bib{AF}{book}{ title = {Rings and Categories of Modules},
  publisher = {Springer},
  year = {1992},
  author = { F. W. Anderson}, author={ K. R. Fuller},
  address = {New-York},

}


\bib{poorNE}{article}{
   author={Er, N.},
   author={L{\'o}pez-Permouth, S.},
   author={S{\"o}kmez, N.},
   title={Rings whose modules have maximal or minimal injectivity domains},
   journal={J. Algebra},
   volume={330},
   date={2011},
   pages={404--417},
  }

\bib{CompletelyDecomposableGroups}{article}{
   author={Fuchs, L.},
   author={Kert{\'e}sz, A.},
   author={Szele, T.},
   title={Abelian groups in which every serving subgroup is a direct
   summand},
   journal={Publ. Math. Debrecen},
   volume={3},
   date={1953},
   pages={95--105 (1954)},

}

\bib{FuchsAbelvol1}{book} {
    AUTHOR = {Fuchs, L.},
     TITLE = {Infinite abelian groups. {V}ol. {I}},
    SERIES = {Pure and Applied Mathematics, Vol. 36},
 PUBLISHER = {Academic Press, New York-London},
      YEAR = {1970},
     PAGES = {xi+290},

}



\bib{FuchsVol.2}{book}{
    AUTHOR = {Fuchs, L.},
     TITLE = {Infinite abelian groups. {V}ol. {II}},
      NOTE = {Pure and Applied Mathematics. Vol. 36-II},
 PUBLISHER = {Academic Press, New York-London},
      YEAR = {1973},
     PAGES = {ix+363},

}


\bib{pi-poorSergio}{article}{
    AUTHOR = {Harmanc{\i}, A.},  author ={  L{\'o}pez-Permouth,
              S.}, author ={ \"{U}ng\"{o}r, B.},
     TITLE = {On the pure-injectivity profile of a ring},
   JOURNAL = {Comm. Algebra},
  FJOURNAL = {Communications in Algebra},
    VOLUME = {},
      YEAR = {to appear},
    NUMBER = {},
     PAGES = {},
}


\bib{ContinuousAndDiscreteModules}{book}{
    AUTHOR = {Mohamed, S. H.}, author ={ M{\"u}ller, B. J.},
     TITLE = {Continuous and discrete modules},
    SERIES = {London Mathematical Society Lecture Note Series},
    VOLUME = {147},
 PUBLISHER = {Cambridge University Press, Cambridge},
      YEAR = {1990},
     PAGES = {viii+126},
      ISBN = {0-521-39975-0},

}


\end{biblist}
\end{bibdiv}
\end{document}